\documentclass[11pt]{article}

\usepackage[latin1]{inputenc}
\usepackage[english]{babel}
\usepackage{here,enumerate,graphicx}
\usepackage{amsmath,amssymb,amscd,amsthm,mathrsfs}

\usepackage[flushmargin]{footmisc}
\newcommand\blfootnote[1]{%
	\begingroup
	\renewcommand\thefootnote{}\footnote{#1}%
	\addtocounter{footnote}{-1}%
	\endgroup}

\theoremstyle{plain}
\newtheorem{theorem}{Theorem}[section]
\newtheorem{proposition}[theorem]{Proposition}

\newtheorem{lemma}[theorem]{Lemma}

\newtheorem{fact}{Fact}[subsection]

\theoremstyle{definition}
\newtheorem{definition}[theorem]{Definition}
\newtheorem{remark}[theorem]{Remark}

\newtheorem{question}[theorem]{Question}

\newtheorem{problems}[theorem]{Problems}

\newcommand{\NN}{\mathbb{N}}
\newcommand{\RR}{\mathbb{R}}
\newcommand{\CC}{\mathbb{C}}
\newcommand{\KK}{\mathbb{K}}
\newcommand{\TT}{\mathbb{T}}

\newcommand{\AP}{\mathcal{AP}}
\newcommand{\Ec}{\mathcal{E}}
\newcommand{\Nc}{\mathcal{N}}
\newcommand{\Uc}{\mathcal{U}}

\newcommand{\URec}{\operatorname{URec}}
\newcommand{\FRec}{\operatorname{FRec}}
\newcommand{\UFRec}{\operatorname{UFRec}}
\newcommand{\RRec}{\operatorname{RRec}}
\newcommand{\Rec}{\textup{Rec}}

\newcommand{\HC}{\textup{HC}}

\newcommand{\cl}{\overline}
\newcommand{\orb}{\mathcal{O}}
\newcommand{\lspan}{\operatorname{span}}
\newcommand{\res}{\arrowvert}
\newcommand{\eps}{\varepsilon}

\newcommand{\Bdsup}{\overline{\operatorname{Bd}}}
\newcommand{\dsup}{\overline{\operatorname{dens}}}
\newcommand{\dinf}{\underline{\operatorname{dens}}}

\newcommand{\sbf}{\boldsymbol}

\newcommand{\ep}[2]{\langle #1 , #2 \rangle}
\newcommand{\Ker}{\textup{Ker}}
\newcommand{\til}{\widetilde}

\begin{document}
\begin{center}
	\begin{LARGE}
		{\bf Two remarks on the set of recurrent vectors}
	\end{LARGE}
\end{center}

\begin{center}
	\begin{Large}
		Antoni L\'opez-Mart\'inez \& Quentin Menet\blfootnote{\textbf{2000 Mathematics Subject Classification}: 47A16, 37B20, 46B87.\\ \textbf{Key words and phrases}: Linear dynamics, hypercyclicity, recurrence, cyclicity, reiterative recurrence.\\ \textbf{Journal-ref}:  	Journal of Mathematical Analysis and Applications, Volume 541, Issue 1, 1 January 2025, 128686.\\ \textbf{DOI}: https://doi.org/10.1016/j.jmaa.2024.128686}
	\end{Large}
\end{center}

\begin{abstract}
	We solve in the negative two open problems, related to the linear and topological structure of the set of recurrent vectors, asked by Sophie Grivaux, Alfred Peris and the first author of this paper. Firstly, we show that there exist recurrent operators whose set of recurrent vectors is not dense lineable; and secondly, we construct operators which are reiteratively recurrent and cyclic, but whose set of reiteratively recurrent vectors is meager.
\end{abstract}

\section{Introduction}

Given a {\em continuous linear operator} $T:X\longrightarrow X$ on a {\em separable infinite-dimensional Banach space} $X$, a vector $x \in X$ is called {\em hypercyclic for $T$} if its orbit,
\[
\orb_T(x) := \{ T^nx \ ; \ n\geq 0 \},
\]
is a dense set in $X$; and $T$ is said to be a {\em hypercyclic operator} whenever it admits a hypercyclic vector. In {\em Linear Dynamics} this property has been investigated in many different directions, one of them being to study the structure of the set $\HC(T)$, that is, the set of hypercyclic vectors for $T$. For instance: it is well-known (due to Birkhoff) that the set $\HC(T)$ is always a {\em dense $G_{\delta}$-set} for any hypercyclic operator $T$ (see \cite[Theorem~2.19]{GrPe2011_book}); and we also know (due to Herrero and Bourdon) that such a set is always {\em dense lineable}, that is, every hypercyclic operator $T$ admits a dense vector subspace that consists (except for the zero-vector) entirely of hypercyclic vectors (see \cite[Theorem~2.55]{GrPe2011_book}).

Another property that has appeared in the last years in the context of Linear Dynamics, and the one in which we are interested here, is that of recurrence: a vector $x \in X$ is called {\em recurrent for $T$} if it belongs to the closure of its forward orbit, that is, if
\[
x \in \cl{\orb_T(Tx)} = \cl{ \{ T^nx \ ; \ n\geq 1\} }.
\]
This definition is equivalent, in our ``{\em Banach space setting}'', to any of the following two facts:
\begin{enumerate}[--]
	\item there exists an increasing sequence of positive integers $(n_k)_{k\in\NN}$ such that $T^{n_k}x \to x$ as $k \to \infty$;
	
	\item for every neighbourhood $U$ of $x$ there exists a positive integer $n \geq 1$ such that $T^nx \in U$.
\end{enumerate}
Moreover, an operator $T$ is said to be a {\em recurrent operator} if the set of recurrent vectors for $T$, which will be denoted by $\Rec(T)$, is a dense set in $X$. Recurrence has a very long history in the context of {\em non-linear dynamical systems} (see \cite{Furstenberg1981_book} and \cite{GotHed1955_book}), while the start of its study in Linear Dynamics can just be dated back to 2014 when the work of Costakis, Manoussos and Parissis \cite{CoMaPa2014} was published. This paper was later followed by the very recent works \cite{BoGrLoPe2022}, \cite{CarMur2022_IEOT,CarMur2022_MS,CarMur2022_arXiv} and \cite{GriLo2023,GriLoPe2022,Lopez2022} among others, and the ``{\em novelty}'' of this property in the linear setting justifies, somehow, the many open problems that there exist right now in {\em linear recurrence}.

About the structure of the set of recurrent vectors many things are known by now: when $T$ is recurrent then the set $\Rec(T)$ is always a {\em dense $G_{\delta}$-set} by \cite[Proposition~2.1]{CoMaPa2014} and a {\em lineable} set, i.e.\ every recurrent operator $T$ admits an infinite-dimensional vector subspace that consists entirely of recurrent vectors (see \cite[Section~5]{GriLoPe2022}). Also the {\em spaceability} of such a set has been studied and, even though hypercyclicity is a much stronger notion than recurrence, the following curious result was proved in \cite{Lopez2022}: {\em a weakly-mixing operator acting on a Banach space admits a closed infinite-dimensional subspace of recurrent vectors if and only if it admits a closed infinite-dimensional subspace that consists (except for the zero-vector) of hypercyclic vectors}. In this paper we will focus on the {\em dense lineability} property for the set $\Rec(T)$, but the reader interested in the notions of lineability and spaceability in a more general context can refer to the book \cite{ABerPeSe2016_book}.

Given an operator $T:X\longrightarrow X$ we say that the set of recurrent vectors $\Rec(T)$ is {\em dense lineable} if it contains a dense vector subspace. This notion has been studied by Grivaux, Peris and the first author of this paper, and it was established in \cite{GriLoPe2022} that a sufficient condition for $\Rec(T)$ to be dense lineable is that of quasi-rigidity: an operator $T$ is called {\em quasi-rigid} if the $N$-fold direct sum operator
\[
\underbrace{T\oplus\cdots\oplus T}_N : \underbrace{X\oplus\cdots\oplus X}_N \longrightarrow \underbrace{X\oplus\cdots\oplus X}_N \ ,
\]
acting as $(x_1,x_2,...,x_N) \longmapsto (Tx_1,Tx_2,...,Tx_N)$, is again a recurrent operator for every $N \in \NN$. It is also shown in \cite[Section~3]{GriLoPe2022} that there exist recurrent operators which are not quasi-rigid, but for the examples constructed there every vector is recurrent, so that they trivially have a dense lineable set of recurrent vectors. One can thus wonder if (just the notion of) recurrence is enough to imply the mentioned dense lineability, as it was asked in \cite[Section~6]{GriLoPe2022}, and as it is the case for hypercyclicity. This is the first open problem that we solve here in the negative (see Section~\ref{Sec_2:not.dense.lineable} below):\\[-5pt]

\begin{question}[\textbf{\cite{GriLoPe2022}}]\label{Q:dense.lineable}
	Let $T:X\longrightarrow X$ be a recurrent operator. Is the set $\Rec(T)$ dense lineable?
\end{question}

In order to state the second problem that we are about to solve, let us introduce a strengthened notion of recurrence called {\em reiterative recurrence}, which appeared in the context of Linear Dynamics for the first time in the recent 2022 paper \cite{BoGrLoPe2022}: a vector $x \in X$ is called {\em reiteratively recurrent for $T$} if the return set from $x$ to any neighbourhood $U$ of $x$, that is, the set
\[
\Nc_T(x,U) := \{ n \geq 1 \ ; \ T^nx \in U \},
\]
has positive upper Banach density, which means that
\begin{equation}\label{eq:Bdsup}
	\Bdsup(\Nc_T(x,U)) := \lim_{N\to\infty} \left( \max_{m\geq 0} \frac{\# \left( \Nc_T(x,U) \cap [m+1,m+N] \right)}{N} \right) > 0.
\end{equation}
We will denote by $\RRec(T)$ the {\em set of reiteratively recurrent vectors for $T$}, which is said to be a {\em reiteratively recurrent operator} whenever $\RRec(T)$ is a dense set. This notion presents a nice relation: {\em an operator is reiteratively recurrent and hypercyclic if and only if it is reiteratively hypercyclic}, which is a strong version of hypercyclicity introduced in \cite{BesMePePu2016} and deeply studied in \cite{BesMePePu2019} and \cite{BoGrLoPe2022}. Note also that
\begin{equation}\label{eq:RRec.Rec}
	\RRec(T) \subset \Rec(T)
\end{equation}
since a vector $x$ belongs to $\Rec(T)$ if and only if the return set $\Nc_T(x,U)$ is non-empty for every neighbourhood $U$ of $x$. Moreover, the usual formula to compute the {\em upper Banach density} for a set of positive integers $J \subset \NN$ is written with a superior limit, that is,
\[
\Bdsup(J) := \limsup_{N\to\infty} \left( \max_{m\geq 0} \frac{\# \left( J \cap [m+1,m+N] \right)}{N} \right),
\]
but the limit is known to exist (see for instance \cite{GreTomTom2010}), so that we can use the formula stated in \eqref{eq:Bdsup}.\newpage

In view of the inclusion \eqref{eq:RRec.Rec} and since $\Rec(T)$ is always a dense $G_{\delta}$-set, it is natural to ask if $\RRec(T)$ is also co-meager for every reiteratively recurrent operator. It was proved in \cite[Theorem~2.1]{BoGrLoPe2022} that this is the case whenever $T$ is also hypercyclic:\\[-17.5pt]
\begin{enumerate}[--]
	\item \cite[Theorem~2.1]{BoGrLoPe2022}: {\em If $T$ is reiteratively recurrent and hypercyclic, then $\RRec(T)$ is co-meager.}\\[-17.5pt]
\end{enumerate}
However, it was also shown in \cite[Example~2.4]{BoGrLoPe2022} that there exist operators $T$ for which the set $\RRec(T)$ can be dense and meager at the same time. By \cite[Theorem~2.1]{BoGrLoPe2022} it is clear that the mentioned examples are non-hypercyclic, but it can be checked that they are even non-cyclic. The next question, which we also solve here in the negative (see Section~\ref{Sec_3:RRec+cyclic.but.meager} below), was then posed in \cite[Problem~5.14]{GriLoPe2022}:\\[-5pt]

\begin{question}[\textbf{\cite{GriLoPe2022}}]\label{Q:RRec.co-meager}
	Let $T:X\longrightarrow X$ be reiteratively recurrent and cyclic. Is $\RRec(T)$ co-meager?
\end{question}

We solve Questions~\ref{Q:dense.lineable} and \ref{Q:RRec.co-meager} in the next sections by constructing counterexamples in every separable infinite-dimensional Banach space. In order to construct such examples we will use the notion of {\em biorthogonal sequence} (see Subsection~\ref{Sec_1.1:notation} below).

The paper is organized as follows: in Section~\ref{Sec_2:not.dense.lineable} we present a modification for the construction of ``{\em recurrent but not quasi-rigid operators}'' shown in \cite{GriLoPe2022}, to exhibit recurrent operators whose set of recurrent vectors is not dense lineable in every separable infinite-dimensional Banach space. This solves Question~\ref{Q:dense.lineable} in the negative. In Section~\ref{Sec_3:RRec+cyclic.but.meager} we construct reiteratively recurrent and cyclic operators whose set of reiteratively recurrent vectors is meager, which solves Question~\ref{Q:RRec.co-meager} in the negative.

We refer the reader to the textbooks \cite{BaMa2009_book} and \cite{GrPe2011_book} for any unexplained notion in Linear Dynamics.

\subsection{Notation for a general separable infinite-dimensional Banach space}\label{Sec_1.1:notation}

We will denote by $\KK$ the field of either real or complex numbers $\RR$ or $\CC$, given any (real or complex) separable infinite-dimensional Banach space $X$, we will denote by $X^*$ its {\em topological dual space}, and given any pair $(x,x^*) \in X\times X^*$ we will denote by $\ep{x^*}{x} := x^*(x)$ the standard dual evaluation.

In the next sections our operators will be built using a bounded biorthogonal sequence. In fact, by a classical result proved in \cite{OvPel1975}, given any separable infinite-dimensional Banach space $X$ we can consider a sequence $(e_k,e_k^*)_{k\in\NN} \subset X\times X^*$ with the following properties:
\begin{enumerate}[$\bullet$]
	\item $\lspan\{ e_k \ ; \ k \in \NN \}$ is dense in $X$;
	
	\item $\ep{e_k^*}{e_j} = \delta_{k,j}$ where $\delta_{k,j}=0$ if $k\neq j$ and $1$ if $k=j$;
	
	\item for each $k \in \NN$ we have that $\|e_k\| = 1$ and $K := \sup_{k\in\NN} \|e_k^*\|^* < \infty$.
\end{enumerate}

We will repeatedly use the fact that given any $x \in X$
\begin{equation}\label{eq:xk<=Kx}
	\left| \ep{e_k^*}{x} \right| \leq K \|x\| \quad \text{ for each } k \in \NN.
\end{equation}

We will always write $c_{00} := \lspan\{ e_k \ ; \ k \in \NN \}$. Note that for any vector $x \in c_{00}$ we have the following equalities
\[
x = \sum_{k \in \NN} \ep{e_k^*}{x} e_k = \sum_{k=1}^{k_x} \ep{e_k^*}{x} e_k \quad \text{ for some } k_x \in \NN.
\]
Note also that, in general, the first equality is false for arbitrary vectors unless the sequence $(e_k)_{k\in\NN}$ is a Schauder basis of the Banach space $X$.

\section{Dense but not dense lineable sets of recurrent vectors}\label{Sec_2:not.dense.lineable}

This section is devoted to show the following result, which solves Question~\ref{Q:dense.lineable}:

\begin{theorem}\label{The:not.dense.lineable}
	Let $X$ be any separable infinite-dimensional Banach space. There exists a recurrent operator $T:X\longrightarrow X$ whose set of recurrent vectors $\Rec(T)$ is not dense lineable.
\end{theorem}

We first assume that $X$ is a \textbf{complex} space and we modify the construction given in \cite[Section~3]{GriLoPe2022}, which was originally based on \cite{Auge2012}. Fix any complex separable infinite-dimensional Banach space $X$ and let $(e_k,e_k^*)_{k\in\NN} \subset X\times X^*$ be a biorthogonal sequence with the properties stated in Subsection~\ref{Sec_1.1:notation}. Grivaux, Peris and the first author of this paper consider in \cite{GriLoPe2022} an operator $T$ given by
\[
Tx := Rx + \sum_{k \geq 3} \frac{1}{m_{k-1}} \ep{w_k^*}{Px} e_k
\]
depending on an operator $R$, a sequence of integers $(m_k)_{k\in\NN}$, a projection $P$ and a sequence of functionals $(w^*_k)_{k\geq 3}$ with bounded norm. Our main modification is letting $(w^*_k)_{k\geq 3}$ to be unbounded.

\subsection{Constructing the operator $T$}

As in \cite{GriLoPe2022} we let $P:X\longrightarrow \lspan\{e_1,e_2\}$ be the projection of $X$ onto the span of $e_1$ and $e_2$ given by
\[
Px := \ep{e_1^*}{x} e_1+\ep{e_2^*}{x} e_2 \quad \text{ for every } x \in X.
\]
Note that $\|P\| \leq 2K$ so that $P$ is continuous. Set $E^* := \lspan\{e_1^*,e_2^*\}$, endowed with the norm $\|\cdot\|^*$ of the dual space $X^*$, and denote by $S_{E^*} := \left\{  w^* \in E^* \ ; \ \|w^*\|^* = 1 \right\}$ the sphere of the 2-dimensional space $E^*$. In \cite{GriLoPe2022} the authors consider for $(w^*_k)_{k\ge 3}$ a dense sequence in $S_{E^*}$, which exists since $S_{E^*}$ is a compact metrizable space, but this choice results in all vectors becoming recurrent for $T$. Since we want the set of recurrent vectors $\Rec(T)$ not to be dense lineable, we select here the functionals $w^*_k$ in such a way to get non-recurrent vectors. To this end, we first consider a dense sequence $(\til{w}_k^*)_{k\geq 3}$ in $S_{E^*}$ and a vector $z\in \lspan\{e_1,e_2\}$ such that 
\begin{equation}\label{eq:z}
	\ep{\til{w}_k^*}{z}\neq 0 \quad \text{ for all } k\geq 3.
\end{equation}
Such a vector $z$ exists since the family $\{\til{w}_k^*\ ;\ k \ge 3\}$ is countable. Given a partition $(A_n)_{n\geq 3}$ of the set $\{ k \in \NN \ ; \ k\geq 3 \}$ with $\# A_n = \infty$ for all $n\geq 3$, we then set
\[
w_k^* := \frac{1}{|\ep{\til{w}_n^*}{z}|} \til{w}_n^* \quad \text{ for each } k \in A_n.
\]
In this way, we have that
\begin{equation}\label{eq:w*.and.z}
	\|w_k^*\|^* \geq \frac{1}{\|z\|} \quad \text{ and } \quad |\ep{w_k^*}{z}|=1 \quad \text{ for every } k\geq 3.
\end{equation}

We now consider a sequence $(m_k)_{k\in\NN} \in \NN^{\NN}$ of positive integers with the following properties:
\begin{enumerate}[(a)]
	\item $m_k \mid m_{k+1}$ for each $k \geq 1$;
	
	\item $m_1 = 1 = m_2$;
\end{enumerate}
and starting from $k=3$, the sequence $(m_k)_{k\geq 3}$ grows fast enough to satisfy:
\begin{enumerate}[(a)]
	\setcounter{enumi}{2}
	\item $\displaystyle \lim_{j\to\infty} \left(m_{j-1} \cdot \sum_{k>j} \frac{\|w_k^*\|^*}{m_{k-1}} \right) = 0$.
\end{enumerate}

These properties are comparable to the properties required in \cite{GriLoPe2022}. The only difference relies on the last condition where we need to take into account the norm of $w_k^*$. The rest of the construction is similar: for each  $x \in c_{00} = \lspan\{ e_k \ ; \ k \in \NN \}$ we set
\[
Rx := \sum_{k\in\NN} \lambda_k \ep{e_k^*}{x} e_k,
\]
where $\lambda_k := \exp(2\pi i \frac{1}{m_k})$ for each $k \in \NN$. By \eqref{eq:xk<=Kx} we have for every $x \in c_{00}$ that
\[
\|Rx\| \leq \|Rx-x\| + \|x\| \leq \sum_{k\in\NN} |\lambda_k-1| \cdot \left| \ep{e_k^*}{x} \right| + \|x\| \leq \left( K \sum_{k\in\NN} |\lambda_k-1| + 1 \right) \|x\|
\]
and using that $|\exp(i\theta)-1| \leq |\theta|$ for every $\theta \in \RR$ we have that
\[
\sum_{k\in\NN} |\lambda_k-1| \leq 2\pi \sum_{k\in\NN} \frac{1}{m_k} < \infty,
\]
since \eqref{eq:w*.and.z} together with condition (c) on the sequence $(m_k)_{k\in\NN}$ show that the series $\sum_{k\in\NN} \frac{1}{m_k}$ is indeed convergent. Then, by the density of $c_{00}$ in the space $X$, the previous inequality also implies that the map $R:c_{00}\longrightarrow c_{00}$ extends to a bounded operator on $X$ still denoted by $R$. Finally, assumption (c) on the sequence $(m_k)_{k\in\NN}$ and the fact that $\|P\| \leq 2K$ imply that 
\[
\left\| \sum_{k\geq 3} \frac{1}{m_{k-1}} \ep{w_k^*}{Px} e_k \right\| \leq \sum_{k\geq 3} \frac{\|w_k^*\|^* \cdot \|Px\|}{m_{k-1}} \leq \left(2K \sum_{k\geq 3} \frac{\|w_k^*\|^*}{m_{k-1}}\right) \|x\|,
\]
so we can define the operator $T$ on $X$ by setting
\begin{equation}\label{eq:def.T}
	Tx := Rx + \sum_{k \geq 3} \frac{1}{m_{k-1}} \ep{w_k^*}{Px} e_k \quad \text{ for every } x \in X.
\end{equation}
It follows that the {\em $n$-th power} of $T$ can be computed exactly as in \cite[Fact~3.3.1]{GriLoPe2022}:

\begin{fact}[\textbf{\cite[Fact~3.3.1]{GriLoPe2022}}]\label{Fact:T.power.n}
	For every $x \in X$ and $n\geq 1$ we have that
	\[
	T^nx = R^nx + \sum_{k\geq 3} \frac{\lambda_{k,n}}{m_{k-1}} \ep{w_k^*}{Px} e_k,
	\]
	where \ $\lambda_{k,n} := \sum_{j=0}^{n-1} \lambda_k^j = \frac{\lambda_k^n-1}{\lambda_k-1}$ \ for each $k \geq 3$.
\end{fact}

We will also need the following properties regarding the numbers $\lambda_{k,n}$:

\begin{fact}[\textbf{\cite[Fact~3.3.2]{GriLoPe2022}}]\label{Fact:lambda_k_n}
	Let $n\geq 1$. Then:
	\begin{enumerate}[{\em(i)}]
		\item $|\lambda_{k,n}| \leq n$ for all $k \geq 3$;
		
		\item $\lambda_{k,m_n}=0$ whenever $n\geq k\geq 3$;
		
		\item $|\lambda_{k,n}|\geq \frac{2}{\pi}n > \frac{m_{k-1}}{\pi}$ whenever $k = \min\{ j \geq 3 \ ; \ 2n \leq m_j \}$.
	\end{enumerate}
\end{fact}

Let us now check that $T$ is a recurrent operator but that $\Rec(T)$ is not dense lineable.

\subsection{Recurrence properties of $T$}

\begin{proposition}\label{Pro:T.recurrent}
	The operator $T:X \longrightarrow X$ is recurrent.
\end{proposition}
\begin{proof}
	Since $\{\til{w}_n^* \ ; \ n \geq 3\}$ is dense in $S_{E^*}$, the union of the kernels $\bigcup_{n \geq 3} \Ker(\til{w}_n^*)$ is dense in the 2-dimensional space $\lspan\{e_1,e_2\}$. Then the set
	\[
	X_0 := \left\{ x\in c_{00} \ ; \ Px \in \bigcup_{n\geq 3} \Ker(\til{w}_n^*) \right\}
	\]
	is dense in $X$ since $X_0$ is dense in $c_{00}$, which is dense in $X$. We claim that $X_0 \subset \Rec(T)$. Indeed, given any $x = \sum_{k=1}^{k_x} \ep{e_k^*}{x} \in X_0$ pick $n\geq 3$ such that $\ep{\til{w}_n^*}{Px}=0$ and let $(k_j)_{j\in\NN}$ be the increasing sequence of integers forming the set $A_n$. By Fact~\ref{Fact:T.power.n} we have that
	\[
	T^{m_{k_j-1}}x - x = \left( R^{m_{k_j-1}}x - x \right) + \sum_{k\geq 3} \frac{\lambda_{k,m_{k_j-1}}}{m_{k-1}} \ep{w_k^*}{Px} e_k,
	\]
	and we will show that this is a $0$-convergent sequence as $j\to\infty$. We start by noticing that, by condition (a) on the sequence $(m_k)_{k\in\NN}$, we have the equality
	\begin{equation}\label{eq:R^.x=x}
		R^{m_{k_j-1}}x = \sum_{k=1}^{k_x} \lambda_k^{m_{k_j-1}} \ep{e_k^*}{x} e_k = \sum_{k=1}^{k_x} \ep{e_k^*}{x} e_k = x \qquad \text{ as soon as } k_j-1 \geq k_x.
	\end{equation}
	Thus, using that $\lambda_{k,m_{k_j-1}} = 0$ for $k_j-1 \geq k$ by (ii) of Fact~\ref{Fact:lambda_k_n} and also the equality $\ep{w_{k_j}^*}{Px} = 0$ for every $k_j\in A_n$, we deduce that
	\begin{equation}\label{eq:Tx-x.simplified}
		\left\| T^{m_{k_j-1}}x - x \right\| = \left\| \sum_{k>k_j} \frac{\lambda_{k,m_{k_j-1}}}{m_{k-1}} \ep{w_k^*}{Px} e_k \right\| \qquad \text{ as soon as } k_j-1\geq k_x.
	\end{equation}
	Moreover, by (i) of Fact~\ref{Fact:lambda_k_n} we have that $|\lambda_{k,m_{k_j-1}}| \leq m_{k_j-1}$ for every $k \in \NN$, and using also the technical condition (c) on the sequence $(m_k)_{k\in\NN}$ we finally get (as soon as $k_j-1\geq k_x$) that
	\[
	\left\| T^{m_{k_j-1}}x - x \right\| \leq \sum_{k > k_j} \frac{|\lambda_{k,m_{k_j-1}}|}{m_{k-1}} \|w_k^*\|^* \cdot \|Px\| \leq 2 K \|x\| \left( m_{k_j-1} \cdot \sum_{k>k_j} \frac{\|w_k^*\|^*}{m_{k-1}}\right) \underset{j\to \infty}{\longrightarrow} 0,
	\]
	which implies that $x \in \Rec(T)$. The density of $X_0$ shows that $T$ is recurrent.
\end{proof}

\begin{proposition}
	The set $\Rec(T)$ is not dense lineable.
\end{proposition}
\begin{proof}
	Let $z$ be the vector considered in \eqref{eq:z}. We start by showing that $P^{-1}(\{z\}) \cap \Rec(T) = \emptyset$. Indeed, let $x\in P^{-1}(\{z\})$, $n\geq 1$ and $k_n := \min\{ j \geq 3 \ ; \ 2n \leq m_j \}$. Note that by (iii) of Fact~\ref{Fact:lambda_k_n}
	\begin{equation}\label{eq:lambda_k_n}
		|\lambda_{k_n,n}| > \tfrac{m_{k_n-1}}{\pi}.
	\end{equation}
	Then
	\begin{eqnarray*}
		\|T^nx-x\| &\geq& \frac{1}{K} \left| \ep{e_{k_n}^*}{T^nx-x} \right| \hspace{4.2cm} \text{ by \eqref{eq:xk<=Kx},}\\[10pt]
		&=& \frac{1}{K} \left| \ep{e_{k_n}^*}{R^nx-x} + \frac{\lambda_{k_n,n}}{m_{k_n-1}} \ep{w_{k_n}^*}{z} \right| \hspace{1.25cm} \text{ by Fact~\ref{Fact:T.power.n},}\\[10pt]
		&>& \frac{1}{K} \left( \frac{1}{\pi} - \left| \ep{e_{k_n}^*}{R^nx-x} \right| \right) \hspace{3cm} \text{by \eqref{eq:w*.and.z} and \eqref{eq:lambda_k_n}.}
	\end{eqnarray*}
	Since $c_{00}$ is dense in $X$ and $e_{k_n}^*$ is continuous we will have, by definition of $R$ on $c_{00}$, that 	
	\[
	\ep{e_{k_n}^*}{R^nx-x}=(\lambda_{k_n}^n-1)\ep{e_{k_n}^*}{x}.
	\]
	Moreover, $\ep{e_{k_n}^*}{x}$ tends to $0$ as $k_n$ tends to infinity because $c_{00}$ is dense in $X$. We deduce that
	\[
	\liminf_{n\to\infty} \left\| T^nx - x \right\| \geq \frac{1}{K\pi},
	\]
	and we conclude that $x \notin \Rec(T)$. Finally, if $\Rec(T)$ contained a dense subspace we would have the equality $P(\Rec(T)) = \lspan\{ e_1, e_2 \}$. This contradicts the fact that $z \notin P(\Rec(T))$.
\end{proof}

The \textbf{complex} version of Theorem~\ref{The:not.dense.lineable} is now proved, but the construction can be easily adapted to the \textbf{real} case using the same arguments as in \cite[Section~3.2]{Auge2012}. It follows that every (real or complex) separable infinite-dimensional Banach space supports a recurrent operator whose set of recurrent vectors is not dense lineable, and Question~\ref{Q:dense.lineable} is now solved.

\begin{remark}\label{Rem:AP.not.dense.lineable}
	The operators constructed in this section fulfill a stronger recurrence notion than the usual one, namely {\em $\AP$-recurrence}:
	\begin{enumerate}[--]
		\item A vector $x \in X$ is called {\em $\AP$-recurrent} for an operator $T:X\longrightarrow X$ if for every neighbourhood $U$ of $x$ the return set $\Nc_T(x,U) = \{ n \geq 1 \ ; \ T^nx \in U \}$ contains arbitrarily long arithmetic progressions; and $T$ is called an {\em $\AP$-recurrent operator} if its set of $\AP$-recurrent vectors, $\AP\Rec(T)$, is dense.
	\end{enumerate}
	In \cite{CarMur2022_MS} it is shown that the inclusions $\RRec(T) \subset \AP\Rec(T) \subset \Rec(T)$ hold for every operator $T$ acting on a Banach space $X$. Moreover, it is also shown that the set $\AP\Rec(T)$ is dense in $X$ if and only if $T$ is {\em topologically multiply recurrent} (see \cite[Proposition~2.2]{CarMur2022_MS}).
	
	For the operators $T$ constructed in this section, the set $X_0$ considered in Proposition~\ref{Pro:T.recurrent} is easily checked to be included in $\AP\Rec(T)$ thanks to condition (c) on the sequence $(m_k)_{k\in\NN}$. Let us quickly argue this fact: fix $\eps>0$ and any length $L \in \NN$, pick any $x = \sum_{k=1}^{k_x} \ep{e_k^*}{x} e_k \in X_0\backslash\{0\}$ and $n\geq 3$ such that $\ep{\til{w}_n^*}{Px}=0$, and let $(k_j)_{j\in\NN}$ be the increasing sequence of integers forming the set $A_n$. Using condition (c) we can choose $j \in \NN$ fulfilling that $k_j-1\geq k_x$ and 
	\[
	m_{k_j-1} \cdot \sum_{k>k_j} \frac{\|w_k^*\|^*}{m_{k-1}} < \frac{\eps}{2K\|x\|L}.
	\]
	Thus, for each $1\leq \ell \leq L$ we have that $R^{\ell \cdot m_{k_j-1}}x = x$ just as in \eqref{eq:R^.x=x} and arguing as in \eqref{eq:Tx-x.simplified} we get that
	\[
	\left\| T^{\ell \cdot m_{k_j-1}}x - x \right\| = \left\| \sum_{k>k_j} \frac{\lambda_{k,\ell \cdot m_{k_j-1}}}{m_{k-1}} \ep{w_k^*}{Px} e_k \right\| \leq 2 K \|x\| \ell \left( m_{k_j-1} \cdot \sum_{k>k_j} \frac{\|w_k^*\|^*}{m_{k-1}}\right) < \eps.
	\]
	This means that the return set $\Nc_T(x,U)$ from $x$ to $U := \{ y \in X \ ; \ \|x-y\|<\eps \}$ contains an arithmetic progression of length $L$, namely $\{ \ell \cdot m_{k_j-1} \ ; \ 1\leq \ell\leq L \}$. The arbitrariness of $\eps$ and $L$ implies that $X_0 \subset \AP\Rec(T)$ and we have even proved the following result, stronger than Theorem~\ref{The:not.dense.lineable}:
	\begin{enumerate}[--]
		\item {\em Every (real or complex) separable infinite-dimensional Banach space $X$ supports an $\AP$-recurrent operator $T:X\longrightarrow X$ for which the set of recurrent vectors $\Rec(T)$, and thus also the set of $\AP$-recurrent vectors $\AP\Rec(T)$, are not dense lineable}.
	\end{enumerate}
\end{remark}

To link to the next section note that, by \cite[Lemma~4.8]{KwiLiOpYe2017}, the set $\AP\Rec(T)$ is always a $G_{\delta}$-set and hence a co-meager set when $T$ is $\AP$-recurrent. This is however not always the case for the set of reiteratively recurrent vectors $\RRec(T)$ even if $T$ is cyclic as we show below.

\section{Dense but not co-meager sets of reiteratively recurrent vectors}\label{Sec_3:RRec+cyclic.but.meager}

In this section we prove the following result, which solves Question~\ref{Q:RRec.co-meager}:

\begin{theorem}\label{The:RRec+cyclic.but.meager}
	Let $X$ be any separable infinite-dimensional Banach space. There exists a reiteratively recurrent cyclic operator $T:X\longrightarrow X$ whose set of reiteratively recurrent vectors $\RRec(T)$ is meager.
\end{theorem}

Recall that a vector $x \in X$ is called {\em cyclic} for an operator $T:X\longrightarrow X$ on a (real or complex) Banach space $X$ if for every non-empty open subset $U \subset X$ there exists a (real or complex) polynomial $p$ such that $p(T)x \in U$; and that $T$ is called a {\em cyclic operator} if it admits a cyclic vector.

\begin{remark}\label{Rem:cyclicity}
	The following is a well-known sufficient condition for $T:X\longrightarrow X$ to be cyclic:
	\begin{enumerate}[--]
		\item {\em For every pair of non-empty open subsets $U,V \subset X$ there exists a (real or complex) polynomial $p$ such that $p(T)(U) \cap V \neq \emptyset$}.
	\end{enumerate}
	Indeed, if this holds and we select a countable base $(U_n)_{n\in\NN}$ of non-empty open sets for $X$, then
	\[
	\bigcap_{n\in\NN} \bigcup_{p \text{ polynomial}} p(T)^{-1}(U_n),
	\]
	is clearly a dense $G_{\delta}$-set of cyclic vectors for $T$.
\end{remark}

\subsection[The family of operators T-lambda-omega]{The family of operators $T_{\sbf{\lambda},\sbf{\omega}}$}

In order to prove Theorem~\ref{The:RRec+cyclic.but.meager} we will use the operators $T_{\sbf{\lambda},\sbf{\omega}} := D_{\sbf{\lambda}} + B_{\sbf{\omega}}$, where $D_{\sbf{\lambda}}$ is a diagonal operator with weights $\sbf{\lambda}=(\lambda_k)_{k\in\NN}$ just as the operator $R$ of Section~\ref{Sec_2:not.dense.lineable}, and where $B_{\sbf{\omega}}$ is the usual unilateral backward shift with weights $\sbf{\omega}=(\omega_k)_{k\in\NN}$. These operators have been considered in the recent work \cite[Chapter~4]{GriMaMe2021_book}, acting on the complex spaces $c_0(\NN)$ and $\ell^p(\NN)$, with the objective of distinguishing the notion of {\em ergodicity} from that of {\em ergodicity in the Gaussian sense}.

Restricted to the vector subspace $c_{00} = \{ e_k \ ; \ k \in \NN \}$ the operator $T_{\sbf{\lambda},\sbf{\omega}}$ can be seen as the following infinite matrix
\[
T_{\sbf{\lambda},\sbf{\omega}} = 
\begin{pmatrix}
	\lambda_1 & \omega_1 & 0 & 0 & \cdots \\
	0 & \lambda_2 & \omega_2 & 0 & \cdots \\
	0 & 0 & \lambda_3 & \omega_3 & \cdots \\
	0 & 0 & 0 & \lambda_4 & \ddots  \\
	\vdots & \vdots  & \vdots & \vdots & \ddots \\
\end{pmatrix} : c_{00} \longrightarrow c_{00} \ , \quad  \left(\ep{e_k^*}{x}\right)_{k\in\NN} \longmapsto \left( \lambda_k\ep{e_k^*}{x} + \omega_k\ep{e_{k+1}^*}{x} \right)_{k\in\NN} \ ,
\]
and $T_{\sbf{\lambda},\sbf{\omega}}$ extends to a continuous operator on $c_0(\NN)$ and $\ell^p(\NN)$ as soon as $\sbf{\lambda}=(\lambda_k)_{k\in\NN}$ and $\sbf{\omega}=(\omega_k)_{k\in\NN}$ are bounded sequences. Since we want to prove the result for a general Banach space we will have to require stronger conditions on the previous sequences in order to guarantee continuity. Indeed, the sequence $\sbf{\lambda}=(\lambda_k)_{k\in\NN}$ will have to converge very fast to $1$ while $\sbf{\omega}=(\omega_k)_{k\in\NN}$ will have to converge very fast to $0$. See Subsection~\ref{Sec_3.2:lambda.omega} for the precise selection of these sequences.

Moreover, and since our objective is constructing a reiteratively recurrent and cyclic operator, we will need some sufficient conditions to guarantee that an operator of the form $T_{\sbf{\lambda},\sbf{\omega}}$ satisfies these dynamical properties. In Lemma~\ref{Lem:eigenvectors+cyclicity} below we will prove a much more general fact regarding the so-called {\em upper-triangular operators}. Let us start by fixing our setting:

\begin{definition}
	Let $X$ be a separable infinite-dimensional Banach space $X$, consider a biorthogonal sequence $(e_k,e_k^*)_{k\in\NN} \subset X\times X^*$ with the properties stated in Subsection~\ref{Sec_1.1:notation} and let $T:X\longrightarrow X$ be a continuous linear operator. We say that $T$ is an {\em upper-triangular operator with respect to $(e_k,e_k^*)_{k\in\NN}$} if the restriction
	\[
	T\res_{c_{00}}:c_{00}\longrightarrow c_{00},
	\]
	of the operator $T$ to the vector subspace $c_{00}=\lspan\{ e_k \ ; \ k \in \NN \}$, can be written as an upper-triangular infinite matrix, that is
	\[
	T\res_{c_{00}} = 
	\begin{pmatrix}
		\lambda_1 & & & & \\
		 & \lambda_2 & & (*) & \\
		 & & \lambda_3 & & \\
		 & (0) & & \lambda_4 & \\
		 & & & & \ddots \\
	\end{pmatrix}.
	\]
	The sequence $\sbf{\lambda}=(\lambda_k)_{k\in\NN} \in \KK^{\NN}$ will be called the {\em diagonal of $T$}.
\end{definition}

\begin{lemma}\label{Lem:eigenvectors+cyclicity}
	Let $X$ be any separable infinite-dimensional Banach space and consider a biorthogonal sequence $(e_k,e_k^*)_{k\in\NN} \subset X\times X^*$ with the properties stated in Subsection~\ref{Sec_1.1:notation}. Suppose that $T:X\longrightarrow X$ is an upper-triangular operator with respect to $(e_k,e_k^*)_{k\in\NN}$, denote by $\sbf{\lambda}=(\lambda_k)_{k\in\NN} \in \KK^{\NN}$ the diagonal of the operator $T$ and assume that $\lambda_k \neq \lambda_l$ for every $k\neq l \in \NN$. Then:
	\begin{enumerate}[{\em(a)}]
		\item The vector subspace $\lspan\left(\bigcup_{k \in \NN} \Ker(T - \lambda_k I) \right)$ is dense in $X$.
		
		\item The set of cyclic vectors for $T$ is co-meager in $X$.
	\end{enumerate}
\end{lemma}
\begin{proof}
	For each $N \in \NN$ set $X_N := \lspan\{ e_k \ ; \ 1\leq k\leq N \}$, which is a $T$-invariant subspace isomorphic to $\KK^N$ by the upper-triangular condition, and consider the restriction $T_N := T\res_{X_N} : X_N \longrightarrow X_N$. It is then enough to check that the following statements hold:
	\begin{enumerate}[(a')]
		\item We have the equality $X_N = \lspan\left(\bigcup_{1\leq k\leq N} \Ker(T_N - \lambda_k I) \right)$ for all $N \in \NN$.
		
		\item The operator $T_N$ has a dense set of cyclic vectors in $X_N$ for all $N \in \NN$.
	\end{enumerate}
	In fact, if (a') holds then statement (a) follows since $\bigcup_{N \in \NN} X_N$ is dense in $X$. Moreover, and since $X_N \subset X_{N+1}$ for all $N \in \NN$, we know that given two non-empty open subsets $U,V \subset X$ we can find $N \in \NN$ such that $U_N:=U \cap X_N$ and $V_N:=V \cap X_N$ are non-empty open subsets of $X_N$. Hence, if statement (b') holds there exists a vector $x \in U_N \subset U$ and a polynomial $p$ such that $p(T)x = p(T_N)x \in V_N \subset V$, so that $p(T)(U) \cap V \neq \emptyset$ and (b) follows from Remark~\ref{Rem:cyclicity}.
	
	Let us now check (a') and (b'): for any fixed $N \in \NN$ we have that
	\[
	T_N =
	\begin{pmatrix}
		\lambda_1 & & & & \\
		0 & \lambda_2 & & (*) & \\
		0 & 0 & \lambda_3 & & \\
		\vdots & \vdots  & \vdots & \ddots & \\
		0 & 0 & 0 & 0 & \lambda_N \\
	\end{pmatrix}
	\]
	and since all the $\lambda_k$ are assumed to be different we deduce that $\sigma_p(T_N)=\{\lambda_k \ ; \ 1\leq k\leq N\}$ and the matrix $T_N$ is similar to the diagonal matrix $D_N = \textup{Diag}(\lambda_1,\lambda_2,...,\lambda_N)$, i.e.\ $T_N = L D_N L^{-1}$ for some invertible operator $L:\KK^N\longrightarrow \KK^N$. Starting with (a'), it is clear that $e_k \in \Ker(D_N-\lambda_k I)$ for each index $1\leq k\leq N$, so that
	\[
	\textstyle X_N = \lspan\left(\bigcup_{1\leq k\leq N} \Ker(D_N - \lambda_k I) \right).
	\]
	It is then immediate to check that $Le_k \in \Ker(T_N-\lambda_k I)$, and since $L$ is invertible we deduce that the set $\{ Le_k \ ; \ 1\leq k\leq N \}$ is an algebraic basis of $X_N$, which finally shows (a'). Regarding (b'), we claim that every vector $x \in X_N$ with $\ep{e_k^*}{x} \neq 0$ for all $1\leq k\leq N$ is cyclic for $D_N$: if such a vector $x$ was not cyclic there would be a non-zero polynomial $p$ of degree less or equal to $N-1$ fulfilling that $0 = p(D_N)x = \sum_{k=1}^N p(\lambda_k)\ep{e_k^*}{x} e_k$, which would imply that $p(\lambda_k)=0$ for every $1\leq k\leq N$, contradicting the maximum number of roots that $p$ can have. Thus, $D_N$ has a dense set of cyclic vectors in $X_N$ and (b') follows since the map $L$ has dense range and the $L$-image of every cyclic vector for the matrix $D_N$ is cyclic for the matrix $T_N$.
\end{proof}

\begin{remark}\label{Rem:RRec+cyclic}
	Suppose now that $X$ is a complex space and $\TT := \{ z \in \CC \ ; \ |z|=1 \}$. If $\sbf{\lambda}=(\lambda_k)_{k\in\NN}$ and $\sbf{\omega}=(\omega_k)_{k\in\NN}$ are sequences of complex numbers fulfilling that the operator $T_{\sbf{\lambda},\sbf{\omega}}:c_{00}\longrightarrow c_{00}$ extends continuously to $X$, it follows from the previous result that:
	\begin{enumerate}[--]
		\item {\em A sufficient condition for the operator $T_{\sbf{\lambda},\sbf{\omega}}$ to be reiteratively recurrent and cyclic is that the sequence of complex numbers $\sbf{\lambda}=(\lambda_k)_{k\in\NN}$ belongs to $\TT^{\NN}$ and $\lambda_k \neq \lambda_l$ for all $k\neq l \in \NN$.}
	\end{enumerate}
	Cyclicity follows directly from statement (b) of Lemma~\ref{Lem:eigenvectors+cyclicity}. For reiterative recurrence recall that $x \in X \setminus \{0\}$ is called a {\em unimodular eigenvector for $T$} provided that $Tx=\lambda x$ for some $\lambda \in \TT$, so that if
	\[
	\Ec(T) := \{ x \in X\setminus\{0\} \ ; \ Tx=\lambda x \text{ for some } \lambda \in \TT \},
	\]
	and if $\sbf{\lambda}=(\lambda_k)_{k\in\NN} \in \TT^{\NN}$, we then have that $\lspan(\Ec(T))$ is dense in $X$ by statement (a) of Lemma~\ref{Lem:eigenvectors+cyclicity}. Finally, it is well-known that $\lspan(\Ec(T)) \subset \RRec(T)$; see for instance \cite{BoGrLoPe2022} or \cite{GriLo2023}. In addition, if we choose each $\lambda_k$ to be a root of unity, then $\lspan\left(\bigcup_{k \in \NN} \Ker(T - \lambda_k I) \right)$ is even formed by periodic vectors for $T$; see \cite[Proposition~2.33]{GrPe2011_book}. This will be the case in our example.
\end{remark}

\subsection[Choosing lambda and omega]{Choosing $\sbf{\lambda}=(\lambda_k)_{k\in\NN}$ and $\sbf{\omega}=(\omega_k)_{k\in\NN}$}\label{Sec_3.2:lambda.omega}

We are finally ready to prove Theorem~\ref{The:RRec+cyclic.but.meager} and we start by the \textbf{complex} case as we did in Section~\ref{Sec_2:not.dense.lineable}: fix any complex separable infinite-dimensional Banach space $X$ and let $(e_k,e_k^*)_{k\in\NN} \subset X\times X^*$ be a biorthogonal sequence with the properties stated in Subsection~\ref{Sec_1.1:notation}. We are going to construct a reiteratively recurrent and cyclic operator $T_{\sbf{\lambda},\sbf{\omega}}$ for which the set $\RRec(T_{\sbf{\lambda},\sbf{\omega}})$ is meager. Recall that by the already mentioned result \cite[Theorem~2.1]{BoGrLoPe2022} such an operator cannot be hypercyclic, and in fact, we will get the non-hypercyclicity by considering some $\omega_k=0$.

We start by fixing $\sbf{\omega}$: for each $j \in \NN$ consider the one-dimensional rank operator
\[
\left( e_{2j-1} \otimes e_{2j}^* \right) : X \longrightarrow X \ , \quad x \longmapsto \ep{e_{2j}^*}{x} e_{2j-1} \ ,
\]
choose any summable sequence $v = (v_j)_{j\in\NN} \in \ell^1(\NN)$ with $|v_j|>0$ for every $j \in \NN$, and consider a sequence $\sbf{\omega}=(\omega_k)_{k\in\NN} \in \CC^{\NN}$ fulfilling that
\[
0 < |\omega_{2j-1}| \leq \frac{|v_j|}{\|e_{2j-1} \otimes e_{2j}^*\|} \qquad \text{ and } \qquad \omega_{2j} = 0
\]
for each $j\geq 1$. Thus the linear map $(\sum_{j \in \NN} \omega_{2j-1} \cdot ( e_{2j-1} \otimes e_{2j}^* ) ) : c_{00} \longrightarrow c_{00}$, which coincides with the backward shift $B_{\sbf{\omega}}$ on $c_{00}$, extends continuously to the whole space $X$ as a compact operator still denoted by $B_{\sbf{\omega}}$ since
\[
\left\| B_{\sbf{\omega}} \right\| = \left\| \sum_{j \in \NN} \omega_{2j-1} \cdot \left( e_{2j-1} \otimes e_{2j}^* \right) \right\| \leq  \sum_{j \in \NN} |\omega_{2j-1}| \cdot \left\| e_{2j-1} \otimes e_{2j}^* \right\| \leq \|v\|_1.
\]

We now fix $\sbf{\lambda}$: by using a trivial recursive process one can construct a strictly increasing sequence of positive integers $(m_j)_{j\in\NN} \in \NN^{\NN}$ fulfilling the properties
\begin{enumerate}[(a)]
	\item $m_j > j$ for every $j \in \NN$;
	
	\item $\displaystyle \lim_{j\to\infty} \frac{1}{m_j |\omega_{2j-1}|} = 0$.
\end{enumerate}
Let now
\[
\lambda_{2j-1} := \exp\left(2\pi i \cdot \frac{1}{m_j^2}\right) \qquad \text{ and } \qquad \lambda_{2j} := \exp\left(2\pi i \cdot \frac{2}{m_j^2}\right)
\]
for each $j \in \NN$. Note that the diagonal linear map
\[
D_{\sbf{\lambda}} : c_{00} \longrightarrow c_{00} \ , \quad x \longmapsto \sum_{k\in\NN} \lambda_k \ep{e_k^*}{x} e_k \ ,
\]
extends to the whole space $X$ (still denoted by $D_{\sbf{\lambda}}$) just as it happened for $R$ in Section~\ref{Sec_2:not.dense.lineable} since
\[
\|D_{\sbf{\lambda}}x\| \leq \|D_{\sbf{\lambda}}x-x\| + \|x\| \leq \sum_{k\in\NN} |\lambda_k-1| \cdot \left| \ep{e_k^*}{x} \right| + \|x\| \leq \left( K \sum_{k\in\NN} |\lambda_k-1| + 1 \right) \|x\|
\]
for every $x \in c_{00}$, and using again that $|\exp(i\theta)-1| \leq |\theta|$ for every $\theta \in \RR$ we have that
\[
\sum_{k\in\NN} |\lambda_k-1| \leq 6\pi \sum_{j\in\NN} \frac{1}{m_j^2} < \infty
\]
by condition (a) on the sequence $(m_j)_{j\in\NN}$.

We can finally consider the operator $T_{\sbf{\lambda},\sbf{\omega}} := D_{\sbf{\lambda}} + B_{\omega} : X \longrightarrow X$, which is continuous and also upper-triangular with respect to $(e_k,e_k^*)_{k\in\NN}$. Since $(m_j)_{j\in\NN}$ is strictly increasing and $m_1>1$ we clearly have that $\lambda_k \neq \lambda_l$ for every $k\neq l \in \NN$ so that Lemma~\ref{Lem:eigenvectors+cyclicity}, and in particular Remark~\ref{Rem:RRec+cyclic}, ensures that the operator $T_{\sbf{\lambda},\sbf{\omega}}$ is reiteratively recurrent and cyclic for these $\sbf{\lambda}=(\lambda_k)_{k\in\NN}$ and $\sbf{\omega}=(\omega_k)_{k\in\NN}$. The form of our operator is the following
\[
T_{\sbf{\lambda},\sbf{\omega}} = 
\begin{pmatrix}
	A_1 & & & & & \\
	& A_2 & & & (0) & \\
	& & A_3 & & & \\
	& & & \ddots & & \\
	& (0) & & & A_j & \\
	& & & & & \ddots \\
\end{pmatrix}
\]
where for each $j\geq 1$ we are denoting by $A_j$ the $2\times 2$ matrix
\begin{equation}\label{eq:A_j}
	A_j = 
	\begin{pmatrix}
		\exp\left(2\pi i \cdot \frac{1}{m_j^2}\right) & \omega_{2j-1} \\
	0 & \exp\left(2\pi i \cdot \frac{2}{m_j^2}\right) \\
	\end{pmatrix}
\end{equation}
so that $T_{\sbf{\lambda},\sbf{\omega}}$ can be expressed as a direct sum of 2-dimensional cyclic operators $T_{\sbf{\lambda},\sbf{\omega}} = \bigoplus_{j\geq 1} A_j$. Our objective is now proving that the set $\RRec(T_{\sbf{\lambda},\sbf{\omega}})$ is meager (see Proposition~\ref{Pro:RRec.meager} below), and we do it by using the dynamical properties of the matrices $A_j$ (see Lemma~\ref{Lem:coordinate(1,2)} below).

Indeed, given three complex values $\mu_1, \mu_2 , \omega \in \CC$ with $\mu_1\neq\mu_2$ and $\omega \neq 0$, then for the $2\times 2$ matrix
\[
A = 
\begin{pmatrix}
	\mu_1 & \omega \\
	0 & \mu_2 \\
\end{pmatrix}
\]
it is easily checked, inductively, that the $n$-th power of $A$ for each $n \in \NN$ has the form
\[
A^n = 
\begin{pmatrix}
	\mu_1^n & \frac{\mu_1^n-\mu_2^n}{\mu_1-\mu_2} \cdot \omega \\
	0 & \mu_2^n \\
\end{pmatrix}.
\]
This formula allows us to proof the following technical fact, which is the key to complete our objective:

\begin{lemma}\label{Lem:coordinate(1,2)}
	For each positive integer $j\geq 1$ and each $n \in \NN$ of the form $n = \ell \cdot m_j^2 + k$ with $\ell \geq 0$ and $m_j \leq k \leq m_j^2-m_j$, we have that the coordinate $(1,2)$ of the matrix $A_j^n=(A_j)^n$ has modulus
	\[
	\left| A_j^n(1,2) \right| = \frac{| \lambda_{2j-1}^n - \lambda_{2j}^n |}{| \lambda_{2j-1} - \lambda_{2j}|} \cdot |\omega_{2j-1}| \geq \frac{2m_j|\omega_{2j-1}|}{\pi}.
	\]
\end{lemma}
\begin{proof}
	Given $j \geq 2$ and any $n \in \NN$ we have that
	\[
	\frac{\left| A_j^n(1,2) \right|}{|\omega_{2j-1}|} = \frac{\left| \lambda_{2j-1}^n - \lambda_{2j}^n \right|}{\left| \lambda_{2j-1} - \lambda_{2j} \right|} = \frac{\left| \exp\left(2\pi n i \cdot \frac{1}{m_j^2}\right) - 1 \right|}{\left| \exp\left(2\pi i \cdot \frac{1}{m_j^2}\right) - 1 \right|} = \frac{\left| \sin\left(\frac{\pi n}{m_j^2}\right) \right|}{\left| \sin\left(\frac{\pi}{m_j^2}\right) \right|} \geq \frac{m_j^2}{\pi} \cdot \left| \sin\left(\frac{\pi n}{m_j^2}\right) \right|,
	\]
	where we have used that $|\sin(\theta)| \leq |\theta|$ for every $\theta \in \RR$. If now we let $n = \ell \cdot m_j^2 + k$ with $\ell \geq 0$ and $m_j \leq k \leq m_j^2/2$, then
	\[
	\left| \sin\left(\frac{\pi n}{m_j^2}\right) \right| = \left| \sin\left(\frac{\pi k}{m_j^2}\right) \right| \geq \frac{2k}{m_j^2} \geq \frac{2}{m_j},
	\]
	by using that $\sin(\theta) \geq \frac{2}{\pi} \theta$ for each $\theta \in [0,\frac{\pi}{2}]$, so that
	\begin{equation}\label{eq:j<=k<=j^2-j}
		\left| A_j^n(1,2) \right| \geq \frac{m_j^2}{\pi} \cdot \left| \sin\left(\frac{\pi n}{j^2}\right) \right| \cdot |\omega_{2j-1}| \geq \frac{2m_j|\omega_{2j-1}|}{\pi},
	\end{equation}
	for each $n = \ell \cdot m_j^2 + k$ with $\ell \geq 0$ and $m_j \leq k \leq m_j^2/2$. Since the sinus function is symmetric with respect to $\frac{\pi}{2}$ we deduce that \eqref{eq:j<=k<=j^2-j} also holds whenever $m_j^2/2 \leq k \leq m_j^2-m_j$.
\end{proof}

\begin{proposition}\label{Pro:RRec.meager}
	The set $\RRec(T_{\sbf{\lambda},\sbf{\omega}})$ is meager.
\end{proposition}
\begin{proof}
	It is enough to show that there exists a co-meager set $G \subset X$ such that $G \cap \RRec(T_{\sbf{\lambda},\sbf{\omega}}) = \emptyset$. The proof is an adaptation of the argument used in \cite[Example~2.4]{BoGrLoPe2022}. Let
	\[
	G := \left\{ x \in X \ ; \ |\ep{e_{2j}^*}{x}|>\tfrac{1}{m_j|\omega_{2j-1}|} \text{ for infinitely many } j \in \NN \right\}.
	\]
	By condition (b) on the sequence $(m_j)_{j\in\NN}$ and using that $c_{00}$ is dense in $X$, this set $G$ can be written as the intersection of countably many dense open subsets
	\[
	G = \bigcap_{N\in \NN} \bigcup_{j \geq N} \left\{ x \in X \ ; \ |\ep{e_{2j}^*}{x}|>\tfrac{1}{m_j|\omega_{2j-1}|} \right\},
	\]
	which shows that $G$ is a dense $G_{\delta}$-set, and hence co-meager.
	
	Fix a vector $x \in G$, let $\eps=\frac{2}{3\pi}$ and consider the neighbourhood $U = \{ y \in X \ ; \ \|y-x\| < \frac{\eps}{K} \}$ of $x$. Note that since $c_{00}$ is dense in $X$ there is some $k_0 \in \NN$ such that $|\ep{e_k^*}{x}|<\eps$ for every $k\geq k_0$. Thus, by \eqref{eq:xk<=Kx}, we have that if $y \in U$ then $|\ep{e_k^*}{y}|<2\eps$ for every $k\geq k_0$. Moreover, since $x \in G$, there exists an infinite set $J \subset \NN$ such that
	\[
	2j-1 \geq k_0 \quad \text{ and } \quad |\ep{e_{2j}^*}{x}|>\frac{1}{m_j|\omega_{2j-1}|} \quad \text{ for each } j \in J.
	\]
	Now, for each $j \geq 1$ and $n \in \NN$, the density of $c_{00}$ in $X$ shows that
	\[
	\ep{e_{2j-1}^*}{T_{\sbf{\lambda},\sbf{\omega}}^n x} = A_j^n(1,2) \cdot \ep{e_{2j}^*}{x} + \lambda_{2j-1}^n \cdot \ep{e_{2j-1}^*}{x},
	\]
	so that if $j \in J$ and $n = \ell \cdot m_j^2 + k$ for some $\ell \geq 0$ and $m_j \leq k \leq m_j^2-m_j$ we have that
	\[
	\left| \ep{e_{2j-1}^*}{ T_{\sbf{\lambda},\sbf{\omega}}^n x } \right| \geq \left|A_j^n(1,2)\right| \cdot \left|\ep{e_{2j}^*}{x}\right| - \left| \ep{e_{2j-1}^*}{x} \right| > \frac{2}{\pi} - \eps = 2\eps
	\]
	by Lemma~\ref{Lem:coordinate(1,2)}. Then $T_{\sbf{\lambda},\sbf{\omega}}^nx \notin U$ and we have that $\# (\Nc_{T_{\sbf{\lambda},\sbf{\omega}}}(x,U) \cap [m+1,m+m_j^2]) \leq 2m_j$ for every integer $m\geq 0$. We can now compute the limit
	\[
	\Bdsup(\Nc_{T_{\sbf{\lambda},\sbf{\omega}}}(x,U)) = \lim_{J \ni j\to\infty} \left( \sup_{m\geq 0}\frac{\# (\Nc_{T_{\sbf{\lambda},\sbf{\omega}}}(x,U) \cap [m+1,m+m_j^2])}{m_j^2} \right) \leq \lim_{J \ni j \to \infty} \tfrac{2m_j}{m_j^2} = 0,
	\]
	so $x \notin \RRec(T_{\sbf{\lambda},\sbf{\omega}})$. This shows that the set of reiteratively recurrent vectors for $T_{\sbf{\lambda},\sbf{\omega}}$ is meager.
\end{proof}

The \textbf{complex} version of Theorem~\ref{The:RRec+cyclic.but.meager} is now proved, but the construction can be easily adapted to the \textbf{real} case by using a conjugacy argument that we include in the following lines: if $X$ is any real separable infinite-dimensional Banach space, and we again let $(e_k,e_k^*)_{k \in \NN}$ be a biorthogonal sequence with the properties stated in Subsection~\ref{Sec_1.1:notation}, we can consider the linear map
\[
T = 
\begin{pmatrix}
	B_1 & & & & & \\
	& B_2 & & & (0) & \\
	& & B_3 & & & \\
	& & & \ddots & & \\
	& (0) & & & B_j & \\
	& & & & & \ddots \\
\end{pmatrix} : c_{00} \longrightarrow c_{00}
\]
where for each $j\geq 1$ we are denoting by $B_j$ the $4\times 4$ matrix
\begin{equation}\label{eq:B_j}
	B_j = 
	\begin{pmatrix}
		\cos\left(2\pi \cdot \frac{1}{m_j^2}\right) & -\sin\left(2\pi \cdot \frac{1}{m_j^2}\right) & \text{Re}(\omega_{2j-1}) & -\text{Im}(\omega_{2j-1}) \\
		\sin\left(2\pi \cdot \frac{1}{m_j^2}\right) & \cos\left(2\pi \cdot \frac{1}{m_j^2}\right) & \text{Im}(\omega_{2j-1}) & \text{Re}(\omega_{2j-1}) \\
		0 & 0 & \cos\left(2\pi \cdot \frac{2}{m_j^2}\right) & -\sin\left(2\pi \cdot \frac{2}{m_j^2}\right) \\
		0 & 0 & \sin\left(2\pi \cdot \frac{2}{m_j^2}\right) & \cos\left(2\pi \cdot \frac{2}{m_j^2}\right) \\
	\end{pmatrix}.
\end{equation}
If the sequence of non-zero complex numbers $(\omega_{2j-1})_{j\in\NN}$ decreases fast enough, and if $(m_j)_{j\in\NN}$ satisfies condition (a) from Subsection~\ref{Sec_3.2:lambda.omega}, then the map $T$ extends continuously to an operator acting on $X$, still denoted by $T$. We can then show that $T$ is reiteratively recurrent, that if $m_1>2$ then $T$ is cyclic, and that if also condition (b) from Subsection~\ref{Sec_3.2:lambda.omega} is satisfied then $\RRec(T)$ is a meager set.

Indeed, if following Lemma~\ref{Lem:eigenvectors+cyclicity} we set $X_{N}:=\lspan\{ e_k \ ; \ 1\leq k \leq N \}$ for each $N \in \NN$, then for every positive integer $j \in \NN$ we can define the homeomorphism
\[
\phi_j : X_{4j} \longrightarrow \til{X_{2j}} \ , \quad \left(\ep{e_k^*}{x}\right)_{k=1}^{4j} \longmapsto \left( \ep{e_{2k-1}^*}{x} + i \ep{e_{2k}^*}{x} \right)_{k=1}^{2j} \ ,
\]
where $\til{X_{2j}} := X_{2j} + i X_{2j}$ denotes the standard {\em complexification} of the real finite-dimensional Banach subspace $X_{2j} = \lspan\{ e_k \ ; \ 1\leq k \leq 2j \} \subset X$ (see \cite{MoMuPeSe2022}), and it is trivial to check that
\[
\phi_j \circ T_{4j} = T_{\sbf{\lambda},\sbf{\omega},2j} \circ \phi_j
\]
for every $j \in \NN$, where
\[
T_{4j} = \begin{pmatrix}
	B_1 & & & &\\
	& B_2 & & (0) &\\
	& & B_3 & & \\
	& (0) & & \ddots & \\
	& & & & B_j \\
\end{pmatrix} : X_{4j} \longrightarrow X_{4j}, \ T_{\sbf{\lambda},\sbf{\omega},2j} =
\begin{pmatrix}
	A_1 & & & &\\
	& A_2 & & (0) &\\
	& & A_3 & & \\
	& (0) & & \ddots & \\
	& & & & A_j \\
\end{pmatrix} : \til{X_{2j}} \longrightarrow \til{X_{2j}},
\]
and where each $A_j$ is the matrix described in \eqref{eq:A_j}. The previous relations and equalities show that each homeomorphism $\phi_j$ is a {\em conjugacy of dynamical systems} (see \cite[Definition~1.5]{GrPe2011_book}) between the finite-dimensional systems $T_{4j}$ and $T_{\sbf{\lambda},\sbf{\omega},2j}$ for each $j \in \NN$. We claim that:
\begin{enumerate}[--]
	\item \textbf{The equality $X_{4j} = \RRec(T_{4j})$ holds for every $j \in \NN$}: by Lemma~\ref{Lem:eigenvectors+cyclicity} we know that
	\[
	\til{X_{2j}} = \RRec(T_{\sbf{\lambda},\sbf{\omega},2j}),
	\]
	for every $j \in \NN$, and a standard conjugacy argument completes the statement (see \cite[Lemma~4.14]{GriLoPe2022}).
	
	\item \textbf{If $m_1>2$ then every $T_{4j}$ has a dense set of cyclic vectors}: using again that $T_{4j}$ is dynamically conjugated to $T_{\sbf{\lambda},\sbf{\omega},2j}$, and hence it is also dynamically conjugated to the diagonal matrix
	\[
	D_{2j} = \text{Diag}(\lambda_1,...,\lambda_{2j}) : \til{X_{2j}} \longrightarrow \til{X_{2j}},
	\]
	it is enough to show that $D_{2j}$ admits a dense set of vectors that are cyclic with respect to the polynomials with real coefficients. We know that the subset of vectors for which every component is a non-zero complex value is dense in $\til{X_{2j}}$, and we can show that these vectors are cyclic with respect to the real polynomials. Indeed, suppose that given such a vector $x$ there was a non-zero polynomial $p$ with real coefficients and of degree less or equal to $4j-1$ such that $p(D_{2j})x = 0$. We would then have that $p(\lambda_k)=0$ for every $1\leq k\leq 2j$. However, since $p$ has real coefficients, then also the conjugate value of every $\lambda_k$ is a root of $p$. Therefore, since all the $\lambda_k$ and $\overline{\lambda_k}$ are different because we assumed $m_1>2$, this contradicts the maximum number of roots that $p$ can have.
\end{enumerate}
Reasoning as in Lemma~\ref{Lem:eigenvectors+cyclicity} we obtain that the whole real-linear operator $T:X\longrightarrow X$ is reiteratively recurrent and cyclic. To show that $\RRec(T)$ is a meager set one can consider in $X$ the dense $G_{\delta}$-set
\[
G := \left\{ x \in X \ ; \ |\ep{e_{4j}^*}{x}|>\tfrac{1}{m_j|\omega_{2j-1}|} \text{ for infinitely many } j \in \NN \right\}
\]
and then check that $G \cap \RRec(T) = \emptyset$ by using Lemma~\ref{Lem:coordinate(1,2)} as in Proposition~\ref{Pro:RRec.meager}. Note that this reasoning follows from the key fact that the matrices $A_j$ as described in \eqref{eq:A_j} are finite-dimensional systems dynamically conjugated to the matrices $B_j$ as described in \eqref{eq:B_j}.

We conclude that every (real or complex) separable infinite-dimensional Banach space supports a reiteratively recurrent and cyclic operator whose set of reiteratively recurrent vectors is meager, and we have finally completely solved Question~\ref{Q:RRec.co-meager}.

\section{Final comments and open problems}

In Section~\ref{Sec_2:not.dense.lineable} we have exhibited the existence of recurrent operators whose set of recurrent vectors is not dense lineable in every separable infinite-dimensional Banach space. As we comment in Remark~\ref{Rem:AP.not.dense.lineable} the same examples are valid for the slightly stronger notion of $\AP$-recurrence. However, the general dense lineability property for other notions such as reiterative, $\Uc$-frequent, frequent or uniform recurrence, is still unknown. Let us introduce these properties and comment on the respective problems:

\begin{definition}[\textbf{\cite{BoGrLoPe2022}}]
	Let $T:X\longrightarrow X$ be a continuous linear operator acting on a Banach space $X$. A vector $x\in X$ is called
	\begin{enumerate}[--]
		\item {\em uniformly recurrent for $T$} if for any neighbourhood $U$ of $x$ the return set
		\[
		\Nc_T(x,U)= \{ n\geq 1 \ ; \ T^nx\in U \}
		\]
		has bounded gaps, that is, there exists $m_U \in \NN$ such that $\Nc_T(x,U) \cap [n,n+m_U] \neq \emptyset$ for all $n \in \NN$. The set of such vectors is denoted by $\URec(T)$ and $T$ is called a {\em uniformly recurrent operator} if such a set is dense in $X$.
		
		\item {\em frequently recurrent for $T$} if for any neighbourhood $U$ of $x$ the return set $\Nc_T(x,U)$, defined as before, has positive lower density, that is,
		\[
		\dinf(\Nc_T(x,U)) = \liminf_{N\to\infty}\frac{\# (\Nc_T(x,U) \cap [1,N])}{N} > 0.
		\]
		The set of such vectors is denoted by $\FRec(T)$ and $T$ is called a {\em frequently recurrent operator} if such a set is dense in $X$.
		
		\item {\em $\Uc$-frequently recurrent for $T$} if for any neighbourhood $U$ of $x$ the return set $\Nc_T(x,U)$, defined as before, has positive upper density, that is,
		\[
		\dsup(\Nc_T(x,U)) = \limsup_{N\to\infty}\frac{\# (\Nc_T(x,U) \cap [1,N])}{N} > 0.
		\]
		The set of such vectors is denoted by $\UFRec(T)$ and $T$ is called a {\em $\Uc$-frequently recurrent operator} if such a set is dense in $X$.
	\end{enumerate}
\end{definition}

As we were mentioning before, the following are open problems:

\begin{problems}
	Let $T:X\longrightarrow X$ be a continuous linear operator acting on a Banach space $X$:
	\begin{enumerate}[(A)]
		\item If $T$ is reiteratively recurrent, is $\RRec(T)$ dense lineable?
		
		\item If $T$ is $\Uc$-frequently recurrent, is $\UFRec(T)$ dense lineable?
		
		\item If $T$ is frequently recurrent, is $\FRec(T)$ dense lineable?
		
		\item If $T$ is uniformly recurrent, is $\URec(T)$ dense lineable?
	\end{enumerate}
\end{problems}

In \cite{BoGrLoPe2022} it is shown that $\URec(T) \subset \FRec(T) \subset \UFRec(T) \subset \RRec(T) \subset \AP\Rec(T) \subset \Rec(T)$ for every operator $T$, so one may wonder if the examples exhibited in Section~\ref{Sec_2:not.dense.lineable} can solve negatively these questions. However, if $\RRec(T)$ is a dense set then $T$ is quasi-rigid and hence $\Rec(T)$ is dense lineable by \cite[Proposition~6.2]{GriLoPe2022}, so that a similar construction/argument to that used here does not apply.

About the strongest recurrence notion between those introduced, namely {\em uniform recurrence}, we know that when the underlying space $X$ is Hilbert then the set $\URec(T)$ is dense lineable as soon as $T$ is uniformly recurrent. This follows from the inclusion $\lspan(\Ec(T)) \subset \URec(T)$, which holds for every operator $T$, together with the fact that
\[
\cl{\lspan(\Ec(T))}=\cl{\URec(T)} \quad \text{ whenever $X$ is a Hilbert space; see \cite[Theorem~1.9]{GriLo2023}}.
\]
However, it is not known if the equality $\cl{\lspan(\Ec(T))}=\cl{\URec(T)}$ is true for an arbitrary operator $T$ acting on general Banach space $X$, so for the moment we can not conclude if $\URec(T)$ is always dense lineable outside the Hilbertian setting.

\section*{Funding}

The first author was supported by the Spanish Ministerio de Ciencia, Innovación y Universidades, grant FPU2019/04094; by MCIN/AEI/10.13039/501100011033, Projects PID2019-105011GB-I00 and PID2022-139449NB-I00; and by the ``Fundaci\'o Ferran Sunyer i Balaguer''. The second author is a Research Associate of the Fonds de la Recherche Scientifique - FNRS.

{

}

{\small
$\ $\\

\textsc{Antoni L\'opez-Mart\'inez}: Universitat Polit\`ecnica de Val\`encia, Institut Universitari de Matem\`atica Pura i Aplicada, Edifici 8E, 4a planta, 46022 Val\`encia, Spain. e-mail: alopezmartinez@mat.upv.es\\

\textsc{Quentin Menet}: Universit\'e de Mons, Service de Probabilit\'es et Statistique, D\'epartement de Math\'ematique, Place du Parc 20, 7000 Mons, Belgium. e-mail: quentin.menet@umons.ac.be
}


\begin{thebibliography}{X}
\addcontentsline{toc}{section}{References}
	
	\bibitem{ABerPeSe2016_book} R. Aron, L. Bernal-Gonz\'{a}lez, D. M. Pellegrino and J. B. Seoane-Sep\'{u}lveda. \textit{Lineability: the search for linearity in mathematics}. Monographs and Research Notes in Mathematics, CRC Press, 2016.
	
	\bibitem{Auge2012} J.-M. Aug\'e. Linear Operators with Wild Dynamics. \textit{Proc. Amer. Math. Soc.}, \textbf{140} (6) (2012), 2103--2116.
	
	\bibitem{BaMa2009_book} F. Bayart and \'E. Matheron. \textit{Dynamics of linear operators}. Cambridge University Press, 2009.
	
	\bibitem{BesMePePu2016} J. B\`es, Q. Menet, A. Peris, and Y. Puig. Recurrence properties of hypercyclic operators. \textit{Math. Ann.}, \textbf{366} (1-2) (2016), 545--572.
	
	\bibitem{BesMePePu2019} J. B\`es, Q. Menet, A. Peris, and Y. Puig. Strong transitivity properties for operators. \textit{J. Differ. Equ.}, \textbf{266} (2-3) (2019), 1313--1337.
	
	\bibitem{BoGrLoPe2022} A. Bonilla, K.-G. Grosse-Erdmann, A. L\'opez-Mart\'inez, and A. Peris. Frequently recurrent operators. \textit{J. Funct. Anal.}, \textbf{283}, Issue 12, 15 December 2022, 109713.
	
	\bibitem{CarMur2022_IEOT} R. Cardeccia and S. Muro. Arithmetic progressions and chaos in linear dynamics. \textit{Integral Equ. Oper. Theory}, \textbf{94} (11) (2022), 18 pages.
	
	\bibitem{CarMur2022_MS} R. Cardeccia and S. Muro. Multiple recurrence and hypercyclicity. \textit{Math. Scand.}, \textbf{128} (3) (2022), 16 pages.
	
	\bibitem{CarMur2022_arXiv} R. Cardeccia and S. Muro. Frequently recurrence properties and block families. Preprint (2022), arXiv:2204.13542.
	
	\bibitem{CoMaPa2014} G. Costakis, A. Manoussos, and I. Parissis. Recurrent linear operators. \textit{Complex Anal. Oper. Theory}, \textbf{8} (2014), 1601--1643.
	
	\bibitem{Furstenberg1981_book} H. Furstenberg. \textit{Recurrence in Ergodic Theory and Combinatorial Number Theory}. Princeton University Press, New Jersey 1981.
	
	\bibitem{GotHed1955_book} W. H. Gottschalk and G. H. Hedlund. \textit{Topological dynamics}. American Mathematical Society Colloquium Publications, volume 36, 1955.
	
	\bibitem{GreTomTom2010} G. Grekos, V. Toma, and J. Tomanov\'a. A note on uniform or Banach density. \textit{Ann. Math. Blaise Pascal}, \textbf{17} (1) (2010), 153--163.
	
	\bibitem{GriLo2023} S. Grivaux and A. L\'opez-Mart\'inez. Recurrence properties for linear dynamical systems: An approach via invariant measures. \textit{J. Math. Pures Appl.}, \textbf{169} (2023), 155--188.
	
	\bibitem{GriLoPe2022} S. Grivaux, A. L\'opez-Mart\'inez and A. Peris. Questions in linear recurrence: From the $T\oplus T$-problem to lineability. Preprint (2022), arXiv:2212.03652.
	
	\bibitem{GriMaMe2021_book} S. Grivaux, \'E. Matheron, and Q. Menet. \textit{Linear dynamical systems on Hilbert spaces: Typical properties and explicit examples}. Memoirs of the AMS, volume 269, 2021.
	
	\bibitem{GrPe2011_book} K.-G. Grosse-Erdmann and A. Peris. \textit{Linear Chaos}. Springer, London 2011.
	
	\bibitem{KwiLiOpYe2017} D. Kwietniak, J. Li, P. Oprocha, and X. Ye. Multi-recurrence and van der Waerden systems. \textit{Science China Mathematics}, \textbf{60} (2017), 59--82.
	
	\bibitem{Lopez2022} A. L\'opez-Mart\'inez. Recurrent subspaces in Banach spaces. Preprint (2022), arXiv.2212.04464.
	
	\bibitem{MoMuPeSe2022} M. S. Moslehian, G. A. Mu\~noz-Fern\'andez, A. M. Peralta, and J. B. Seoane-Sep\'ulveda. Similarities and differences between real and complex Banach spaces: an overview and recent developments. \textit{RACSAM}, \textbf{116} (2) (2022), 80 pages.
	
	\bibitem{OvPel1975} R.I. Ovsepian and A. Pe\l czynski. On the existence of a fundamental total and bounded biorthogonal sequence in every separable Banach space, and related constructions of uniformly bounded orthonormal systems in $L^2$. \textit{Studia Math.}, \textbf{54} (1975), 149--159.

\end{thebibliography}
\end{document}